\documentclass[12pt,centertags,oneside]{article}
\usepackage{amsmath,amstext,amsthm,amscd,typearea}
\usepackage{amssymb}
\usepackage{a4wide}
\usepackage[mathscr]{eucal}
\usepackage{mathrsfs}
\usepackage{typearea}
\usepackage{charter}
\usepackage{pdfsync}
\usepackage{url}

\usepackage{xcolor}

\usepackage[a4paper,width=16.2cm,top=3cm,bottom=3cm]{geometry}

\numberwithin{equation}{section}

\newtheorem{theorem}{Theorem}[section]

\newtheorem{proposition}[theorem]{Proposition}
\newtheorem{corollary}[theorem]{Corollary}
\newtheorem{lemma}[theorem]{Lemma}
\newtheorem{remark}[theorem]{Remark}

\newtheorem{example}[theorem]{Example}

\newcommand{\cali}[1]{\mathscr{#1}}

\newcommand{\supp}{{\rm Supp}}

\newcommand{\ddc}{dd^c}

\newcommand{\id}{{\rm id}}

\newcommand{\C}{\mathbb{C}}

\newcommand{\N}{\mathbb{N}}

\newcommand{\R}{\mathbb{R}}
\renewcommand\P{\mathbb{P}}

\title{\bf  Lelong numbers of currents of full mass intersection}
\providecommand{\keywords}[1]{\textbf{\textit{Keywords:}} #1}
\providecommand{\subject}[1]{\textbf{\textit{Mathematics Subject Classification 2010:}} #1}

\author{Duc-Viet Vu}
\newcommand{\Addresses}{{
		\bigskip
		\footnotesize
		\textsc{Duc-Viet Vu, University of Cologne,  Department of Mathematics and Computer Science, Division of Mathematics, Weyertal 86-90, 50931 K\"oln,  Germany}
		\noindent
		\par\nopagebreak
		\noindent
		\textit{E-mail address}: \texttt{vuviet@math.uni-koeln.de}	
}}

\date{\emph{In memory of Nessim Sibony}}
\begin{document}
\maketitle
\begin{abstract} We study Lelong numbers of currents of full mass intersection on a compact K\"ahler manifold in a mixed setting. Our main theorems cover some recent results due to Darvas-Di Nezza-Lu. The key ingredient in our approach is a new notion of products of pseudoeffective $(1,1)$-classes which captures some ``pluripolar part" of the ``total intersection" of given pseudoeffective $(1,1)$-classes. 
\end{abstract}

\noindent
,  {relative non-pluripolar product}, {full mass intersection}, {Lelong number}.
\\

\noindent
\subject{32U15}, {32Q15}.

\section{Introduction}

Let $X$ be a compact K\"ahler manifold of dimension $n$. For every closed positive current $S$ on $X$, we denote by $\{S\}$ its cohomology class.  For cohomology $(q,q)$-classes $\alpha$ and $\beta$ on $X$, we write $\alpha \le \beta$ if $\beta- \alpha$ can be represented by a closed positive $(q,q)$-current.
    
Let $\alpha_1,\ldots, \alpha_m$ be pseudoeffective $(1,1)$-classes, where $1 \le m \le n$.  Let $T_j$ and $T'_j$ be  closed positive $(1,1)$-currents  in $\alpha_j$  for $1 \le j \le m$ such that $T_j$ is more singular than $T'_j$, \emph{i.e,} potentials of $T_j$ is smaller  than those of $T'_j$ modulo an additive constant.  By monotonicity of non-pluripolar products (see \cite[Theorem 1.1]{Viet-generalized-nonpluri} and also \cite{BEGZ,Lu-Darvas-DiNezza-mono,WittNystrom-mono}), there holds
\begin{align}\label{ine-mono-relativenopluri}
\{\langle T_1 \wedge \cdots \wedge T_m \rangle \} \le \{\langle T_1' \wedge \cdots \wedge T_m' \rangle \}.
\end{align}
We refer to the beginning of Section \ref{sec-positiveproduct} for a brief recap of non-pluripolar products. 

We are interested in comparing the singularity types of $T_j$ and  $T'_j$ when the equality in (\ref{ine-mono-relativenopluri}) occurs. Given the generality of the problem, it is desirable to formulate it in a more concrete way. In what follows, we focus on the important setting where $T_1, \ldots, T_m$ are of full mass intersection (\emph{i.e,} $T'_j$'s have minimal singularities in their cohomology classes).

Let us recall that  $T_1,\ldots, T_m$ are said to be  of \emph{full mass intersection} if the equality  in (\ref{ine-mono-relativenopluri}) occurs for $T'_j$ to be a current with minimal singularities $T_{j,\min}$ in $\alpha_j$ for $1 \le j \le m$. This is independent of the choice of $T_{j,\min}$.  The last notion has played an important role in complex geometry, for example, see \cite{Ahn-NC,BEGZ,Darvas_book,Lu-Darvas-DiNezza-logconcave,DinhSibony_pullback,Koziarz-DucManh,Viet-convexity-weightedclass,WittNystrom-duality}. We also notice that  a connection of the notion of full mass intersection with the theory of density currents (see \cite{Dinh_Sibony_density}) was established in \cite{Viet-density-nonpluripolar}, see also \cite{VietTuanLucas}.  

%Assume now that  $T_1, \ldots, T_m$ are of full mass intersection. 
One of the most basic objects to measure the singularity of a current is the notion of Lelong numbers. We refer to \cite{Demailly_ag} for its basic properties. Hence, the purpose of this paper is to compare the Lelong numbers of $T_j$ and $T_{j,\min}$ when $T_1, \ldots,T_m$ are of full mass intersection.  %Our main result below gives a natural setting where the last question is answered in a somewhat complete form. 
 To go into details, we need some notions. 

Let $S$ be a closed positive current  on $X$ and $x$ be a point in $X$. Denote by $\nu(S,x)$ the Lelong number of $S$ at $x$. One can compute  $\nu(S,x)$ as follows. We write $S= \ddc \psi$ for some psh function $\psi$ defined on an open neighborhood $U$ of $x$ such that $U$ is a local chart of $X$ which we identify with an open subset in $\C^n$ and the point $x$ corresponds to the origin in $\C^n$. Then we have  
$$\nu(S,x)= \max\{\gamma \in \R_{\ge 0}: \psi(z) \le \gamma\log |z| + O(1) \, \text{near } 0\},$$
see \cite[Chapter III]{Demailly_ag}. Let $V$ be  an irreducible analytic subset of $X$. By Siu's analytic semi-continuity of Lelong numbers (\cite{Demailly_ag,Siu}), for every $x\in V$ outside some proper analytic subset of $V$, we have 
$$\nu(S,x)= \min_{x' \in V} \nu(S,x').$$
The last number is called \emph{the generic Lelong number of $S$ along $V$} and is denoted by $\nu(S,V)$. 
  
Let $\alpha$ be a pseudoeffective $(1,1)$-class on $X$. Following \cite{Demailly_analyticmethod}, we  recall that $\alpha$ is said to be \emph{big} if there is a K\"ahler current in $\alpha$, \emph{i.e,} there is a closed positive current $T$ in $\alpha$ such that $T \ge \omega$ for some K\"ahler form $\omega$ on $X$.  Let $T_{\alpha,\min}$ be a current with minimal singularities in $\alpha$ (see \cite[page 41-42]{Demailly_analyticmethod} for definition). We denote by $\nu(\alpha,V)$ the generic Lelong number of $T_{\alpha,\min}$  along $V$. This number is independent of the choice of $T_{\alpha,\min}$. It is clear that for every current $S\in \alpha$, we have $\nu(S,V) \ge \nu(\alpha,V)$.    Here is our first main result.

\begin{theorem}\label{the-main} Let $1\le m \le n$ be an integer. Let $\alpha_1,\ldots, \alpha_m$ be big cohomology classes in $X$ and  let $T_j$ be   a closed positive $(1,1)$-currents  in $\alpha_j$ for $1\le j \le m$. Let $V$ be a proper irreducible analytic subset of $X$ of dimension $\ge n-m$. Assume that  $T_1,\ldots,T_m$ are of full mass intersection. Then there exists an index $1\le j \le m$ such that 
\begin{align}\label{eq-genericLelongbangnhau}
\nu(T_j,V) = \nu(\alpha_j,V).
\end{align}
\end{theorem}   

We note that when $\alpha_1,\ldots, \alpha_m$ are K\"ahler, Theorem  \ref{the-main} was proved in  \cite[Theorem 1.2]{Viet-generalized-nonpluri}; see also the discussion after Corollary \ref{cor-knownbignef} below. The proof presented there is not applicable in the setting of Theorem \ref{the-main}.

When $\dim V= n-m$, the above result is optimal because in general, it might happen that there is only one index $j$ satisfying (\ref{eq-genericLelongbangnhau}); see Example \ref{ex-optimallelong}. However, motivated from the K\"ahler case, we wonder whether it is true that the number of $1 \le j \le m$ such that $\nu(T_j,V)= \nu(\alpha_j,V)$ is at least $\dim V - (n-m)+1$ (recall $V \subsetneq X$).

In the case where $m=n$, our above result can be improved quantitatively as follows. 

\begin{theorem} \label{the-main3} Let $\cali{B}_0$ be a closed cone in $H^{1,1}(X, \R)$ which is contained  in the cone of big $(1,1)$-classes of $X$.  Then, there exists a constant $C>0$ such that for every $x_0 \in X$,  every $\alpha_j \in \cali{B}_0$ and  every closed positive $(1,1)$-current $T_j \in \alpha_j$ for $1\le j \le n$, we have  
\begin{align}\label{ine-truonghopncaliB0}
\int_X \big(\langle \wedge_{j=1}^n\alpha_j \rangle- \{\langle \wedge_{j=1}^n T_j \rangle\}\big) \ge C\big(\nu(T_1, x_0)- \nu(\alpha_1, x_0)\big)\cdots \big(\nu(T_n, x_0)- \nu(\alpha_n, x_0)\big).
\end{align}
\end{theorem} 

%Here $\R_{>0}$ is the set of positive real numbers, and $\R_{>0} \cali{B}_0:= \{\lambda \alpha: \lambda \in \R_{>0}, \alpha \in \cali{B}_0\}$. 
The dependence of $C$ on $\cali{B}_0$ is necessary, see Example \ref{ex-themain3}.  % In the next part of the paper, we are interested in the case where  $\alpha_1=\cdots=\alpha_m$ and $T_1= \ldots=T_m$. We will generalize some results \cite{Lu-Darvas-DiNezza-singularitytype} to the case where $m \le n$. As we will see later, the method in \cite{Lu-Darvas-DiNezza-singularitytype} doesn't work in our setting. Firstly, 
We have the following direct consequences of Theorem \ref{the-main}.

\begin{corollary} \label{cor-main2} Let $1\le m \le n$ be an integer.  Let $\alpha$ be a big class and let $T \in \alpha$ be a closed positive $(1,1)$-current so that 
$$\{\langle T^m \rangle\}= \langle \alpha^m \rangle.$$
Let $V$ be an irreducible analytic subset of $X$ of dimension at least $n-m$.  Then there holds
$$\nu(T,V)=\nu(\alpha_j,V).$$
In particular, if $\alpha$ is big and nef, then $T$ has zero Lelong number at a generic point in $V$. 
\end{corollary}

Recall that $\langle \alpha^m \rangle $ is defined to be the cohomology class of $\langle T_{\alpha,\min}^m \rangle$, where $T_{\alpha,\min}$ is a current with minimal singularities in $\alpha$, see Section \ref{sec-positiveproduct} below for details.  Combining Corollary \ref{cor-main2} with results in \cite{Boucksom_l2,Boucksom-Favre-Jonsson}, we recover the following known result. 

\begin{corollary} \label{cor-knownbignef} Let $\theta$ be a smooth closed $(1,1)$-form in a big cohomology class $\alpha$. Let $\varphi$ be a $\theta$-psh function of full Monge-Amp\`ere mass, \emph{i.e,} 
$$\{\langle (\ddc \varphi+ \theta)^n \rangle\}= \langle \alpha^n \rangle.$$
Let $\varphi_{\alpha,\min}$ be a $\theta$-psh function with minimal singularities. Then, we have 
\begin{align} \label{eq-multiidealsheaf}
\mathcal{I}(t\varphi)= \mathcal{I}(t \varphi_{\alpha,\min})
\end{align}
 for every $t>0$, where for every quasi-psh function $\psi$ on $X$, we denote by $\mathcal{I}(\psi)$ the multiplier ideal sheaf associated to $\psi$.  
\end{corollary}

Corollary \ref{cor-knownbignef} was proved in  \cite{Darvas-kahlerclass,Lu-Darvas-DiNezza-mono,Lu-Darvas-DiNezza-singularitytype} (hence answering a question posed in \cite{DGZ}); see also \cite{GZ-weighted} for the case where $\theta$ is K\"ahler. In fact, \cite{Lu-Darvas-DiNezza-mono} gives a stronger fact which we describe below. For  every closed positive $(1,1)$-current $T'$ with  $\int_X \langle T'^n \rangle >0,$ Theorem 1.3 in \cite{Lu-Darvas-DiNezza-mono} gives a characterization (in terms of certain plurisubharmonic rooftop envelopes) of  potentials of every closed positive $(1,1)$-current $T$ cohomologous to $T'$ such that $T$ is less singular than $T'$ and 
\begin{align*} %\label{eq-giathietTTphay}
\int_X \langle T^n \rangle = \int_X \langle T'^n \rangle.
\end{align*}
 Consequently,  the multiplier ideal sheafs associated to the potentials of $T$ and $T'$ are the same by arguments from the proof of \cite[Theorem 1.1]{Lu-Darvas-DiNezza-singularitytype}.  Nevertheless, in the present setting of our main results, it is unclear how to formulate  such a characterization  because either $T_1,\ldots, T_m$ can be different or  $m\le n$ (even if one takes $T_1= \cdots =T_m$). In fact, a direct analogue of the envelope characterization given \cite{Lu-Darvas-DiNezza-mono} is not true in our setting when $m\le n$; see the comment after Theorem \ref{the-main} and \cite[Remark 3.3]{Lu-Darvas-DiNezza-mono}.

Let us now have a few comments on our approach.  Due to the above discussions, we present here a \emph{completely new strategy} to the study of singularity of currents of full mass intersection.  We stress that although our main results only involve the usual non-pluripolar products, the notion of relative non-pluripolar products introduced in \cite{Viet-generalized-nonpluri} will play an essential role in our proof. The reason, which will be more clear later, is that relative non-pluripolar products allow us to better control the loss of masses. 

The key ingredient in our proof of Theorem \ref{the-main} is a new notion of products of pseudoeffective classes which was briefly mentioned in \cite[Remark 4.5]{Viet-generalized-nonpluri}. This new product of pseudoeffective classes is bounded from below by  the  positive product introduced in \cite{Boucksom-these,BEGZ}. The feature is that this new product also captures some pluripolar part of ``total intersection" of classes. This explains why we have a better control on masses. 

%Although we don't have an example, it is likely that the dependence of $C_r$ on $\cali{B}_0$ is necessary. We don't know what should be the optimal constant for $C$.  %The dependence of the constant $C_r$ on $r$ is necessary because otherwise  we would deduce that  $\delta(x)$ is as small as we want for every $x \in X$ outside a certain finite set of points in $X$. This is not true in general by taking $X$ to be projective, $\alpha_j=\alpha$ and $T_j=T$ for every $j$,  where $\alpha$ is a Chern class of a very ample line bundle on $X$ and $T$ is a current of integration along a divisor in $\alpha$.  
Theorem \ref{the-main3} is a direct consequence of the proof of Theorem \ref{the-main}.  We underline that our arguments in the proof of Theorem \ref{the-main} are not quantifiable as soon as $\dim V \ge 2$. This is due to the fact that we need to use the blowup along $V$ and the desingularization of $V$ (in case $V$ is singular). Despite of this, we think that it is still reasonable to expect an estimate similar to Theorem \ref{the-main3} in the case where $V$ is of higher dimension.  

Finally, in view of the above discussion of results in  \cite{Lu-Darvas-DiNezza-mono}, one can wonder what should be expected for the equality case of (\ref{ine-mono-relativenopluri}) when $T'_j$'s are not necessarily of minimal singularities. It is not unrealistic to hope that our approach can be extended to this setting. But there are non-trivial obstructions. To single out one: the condition that $T'_j$'s have minimal singularities are needed in our proof of Theorem \ref{the-main} because we will use the fact that there are K\"ahler currents with analytic singularities which are more singular  than $T_j$ for every $j$.

The paper is organized as follows. In Section \ref{sec-positiveproduct}, we present basic properties of relative non-pluripolar products and introduce the above-mentioned notion of products of pseudoeffective classes. Theorems \ref{the-main} and \ref{the-main3} are proved in Section \ref{sec-themain12}.
\\

%Finally, we would like to point out that Theorems \ref{the-main} and \ref{the-main-relative} for $m=n-1$  are related to the study of the dual of the cone of pseudoeffective classes, which was conjectured in \cite{???}  to be equal to the cone generated by positive products of $(1,1)$-classes. The last conjecture was confirmed in \cite{Wittnystrom}.  They also have links to the study of the null part ........  When $m=n-1$, the last notion is closely related to a long standing conjecture ???? Cite Witt Nystrom and other heres. When $m=n$, cite complex M-A equations here ???? Maybe also Tossati ???? \\

\noindent
\textbf{Acknowledgments.} We thank Tam\'as Darvas and Tuyen Trung Truong for fruitful discussions.  This research  is supported by a postdoctoral fellowship of the Alexander von Humboldt Foundation. \\

\section{Relative non-pluripolar products} \label{sec-positiveproduct}

We first recall some basic facts about relative non-pluripolar products. This notion was introduced in \cite{Viet-generalized-nonpluri} as a generalization of the usual non-pluripolar products  given in \cite{BT_fine_87,BEGZ,GZ-weighted}. To simplify the presentation, we only consider the compact setting.

Let $X$ be a compact K\"ahler manifold of dimension $n$.  Let $T_1, \ldots, T_m$ be closed positive $(1,1)$-currents on $X$. Let $T$ be  a closed positive current of bi-degree $(p,p)$ on $X$.  By \cite{Viet-generalized-nonpluri}, we can define \emph{the $T$-relative non-pluripolar product} $\langle \wedge_{j=1}^m T_j \dot{\wedge} T\rangle$  in a way similar to that of  the usual non-pluripolar product.  For readers' convenience, we recall how to do it. 

Write $T_j= \ddc u_j+ \theta_j$, where $\theta_j$ is a smooth form and $u_j$ is a $\theta_j$-psh function. Put 
$$R_k:=\bold{1}_{\cap_{j=1}^m \{u_j >-k\}} \wedge_{j=1}^m (\ddc \max\{u_j,-k\} + \theta_j)\wedge T$$
for $k \in \N$.  By the strong quasi-continuity of bounded psh functions (\cite[Theorems 2.4 and 2.9]{Viet-generalized-nonpluri}), we have 
$$R_k= \bold{1}_{\cap_{j=1}^m \{u_j >-k\}} \wedge_{j=1}^m (\ddc \max\{u_j,-l\} + \theta_j)\wedge T$$
for every $l \ge k \ge 1$. A similar equality also holds if we use local potentials of $T_j$ instead of global ones. We can show that  $R_k$ is positive (see \cite[Lemma 3.2]{Viet-generalized-nonpluri}). 

As in \cite{BEGZ}, since $X$ is K\"ahler, one can check  that $R_k$ is of mass bounded uniformly in $k$ and $(R_k)_k$ admits a limit current which is closed as $k \to \infty$. The last limit is denoted by $\langle \wedge_{j=1}^m T_j \dot{\wedge} T\rangle$.  The  last product is, hence,  a well-defined closed positive current of bi-degree $(m+p,m+p)$; and it is  symmetric with respect to $T_1, \ldots, T_m$ and homogeneous. We refer to  \cite[Proposition 3.5]{Viet-generalized-nonpluri} for more properties of relative non-pluripolar products. When $T \equiv 1$, the $T$-relative non-pluripolar product $\langle T_1 \wedge \cdots \wedge T_m \dot{\wedge} T\rangle$ is exactly the non-pluripolar product $\langle T_1 \wedge \cdots \wedge T_m\rangle$ of $T_1,\ldots, T_m$ defined in \cite{BEGZ}.  

Let $\alpha_1,\ldots, \alpha_m$ be pseudoeffective $(1,1)$-classes on $X$. Recall that by using a monotonicity of relative non-pluripolar products (\cite[Theorem 1.1]{Viet-generalized-nonpluri}), we can define the cohomology class $\{\langle \alpha_1 \wedge \ldots \wedge \alpha_m \dot{\wedge} T \rangle \}$ which is the one of the current $\langle \wedge_{j=1}^m T_{j,\min} \dot{\wedge} T \rangle$, where $T_{j,\min}$ is a current with minimal singularities in $\alpha_j$ for $1 \le j \le m$.    When $T$ is the current of integration along $X$, we write $\langle \alpha_1 \wedge \cdots \wedge \alpha_m \rangle$ for $\{\langle \alpha_1\wedge  \cdots \wedge \alpha_m \dot{\wedge} T \rangle\}$.  
By \cite[Proposition 4.6]{Viet-generalized-nonpluri}, the class $\langle \alpha_1\wedge  \cdots \wedge \alpha_m \rangle$ is equal to the positive product of $\alpha_1, \ldots, \alpha_m$ defined in  \cite[Definition 1.17]{BEGZ} provided that  $\alpha_1, \ldots, \alpha_m$ are big. %One can actually check that  when $p=n-m$, the class $\{\langle \alpha_1 \wedge \ldots \wedge \alpha_m \dot{\wedge} T \rangle \}$ is indeed equal to the mobile intersection number of $\alpha_1,\ldots, \alpha_m,T$ given in \cite[Definition 3.2.1]{Boucksom-these}. 

In the next paragraph, we are going to introduce a related notion of products of $(1,1)$-classes. This idea was already suggested in \cite{Viet-generalized-nonpluri}. This new notion will play a crucial role in our proof of Theorem \ref{the-main}.   We are interested in the case where $T$ is of bi-degree $(1,1)$.  We recall the following key monotonicity property. 

\begin{theorem} \label{th-mono-current11} (\cite[Remark 4.5]{Viet-generalized-nonpluri}) Let $X$ be  a compact K\"ahler manifold and let $T_1,\ldots, T_m,T$ be closed positive $(1,1)$-currents on $X$.  Let $T'_j$ and $T'$ be closed positive $(1,1)$-currents in the cohomology class of $T_j$ and $T$ respectively such that $T'_j$ is less singular than $T_j$ for $1 \le j \le m$ and $T'$ is less singular than $T$. Then we have 
$$\{\langle T_1 \wedge \cdots \wedge T_m \dot{\wedge} T \rangle \} \le \{\langle T'_1 \wedge \cdots \wedge T'_m \dot{\wedge} T' \rangle\}.$$
\end{theorem}

Recall that for closed positive $(1,1)$-currents $P$ and $P'$ on $X$, we say that $P'$ is less singular than $P$ if for every global potential $u$ of $P$ and $u'$ of $P'$, then $u \le u' +O(1)$. 

\proof Since this result is crucial for us, we will present its proof below. Write $T_j= \ddc u_j + \theta_j$, $T'_j=\ddc u'_j+ \theta_j$, where $\theta_j$ is  a smooth form  and $u'_j, u_j$ are negative $\theta_j$-psh functions,  for every $1 \le j \le m$. Similarly, we have $T= \ddc \varphi+ \eta$, $T'=  \ddc\varphi'+ \eta'$. \\

\noindent
\textbf{Step 1.}  Assume for the moment that $T_j,T'_j$ are of the same singularity type for every $1 \le j \le m$ and $T,T'$ are also of the same singularity type. We will check that 
\begin{align}\label{eq-same-singularity}
\{\langle T_1 \wedge \cdots \wedge T_m \dot{\wedge} T \rangle \} = \{\langle T'_1 \wedge \cdots \wedge T'_m \dot{\wedge} T' \rangle\}.
\end{align}
Since $T_j,T'_j$ are of the same singularity type, we have $\{u_j= -\infty\}= \{u'_j= -\infty\}$ and  $w_j:= u_j - u'_j$ is bounded. We have similar properties for $\varphi, \varphi'$.    Let $A:= \cup_{j=1}^m \{u_j= -\infty\}$ which is a complete pluripolar set. Put $u_{j k}:= \max\{u_j, -k\}$, $u'_{j k}:= \max\{u'_j, -k\}$ and
\begin{align}\label{equa-defvarphik}
\psi_k:= k^{-1}\max\{\sum_{j=1}^n (u_j+u'_j), -k\}+ 1
\end{align}
which is quasi-psh and $0\le \psi_k \le 1$, $\psi_k(x)$ increases to $1$ for $x\not \in A$. We have $\psi_k(x) =0$ if $u_j(x)\le -k$ or $u'_j(x) \le -k$ for some $j$.  Put $w_{j k}:= u_{j k} - u'_{j k}$. Since $w_j$ is bounded, we have
\begin{align}\label{ine-hieuujujphay2}
|w_{j k}| \lesssim 1
\end{align}
on $X$.   Let $J, J' \subset \{1, \ldots, m\}$ with $J \cap J' = \varnothing$. Put
$$R_{JJ'k}:= \wedge_{j \in J}  (\ddc u_{j k} + \theta_j) \wedge \wedge_{j' \in J'} (\ddc u'_{j' k} + \theta_{j'})\wedge T$$
and
$$R_{JJ'}:= \big \langle \wedge_{j \in J}  (\ddc u_{j} + \theta_j) \wedge \wedge_{j' \in J'} (\ddc u'_{j'} + \theta_{j'}) \dot{\wedge}T\big\rangle.$$
Let 
$$B_k:= \cap_{j \in J} \{u_j > -k\} \cap \cap_{j' \in J'}\{ u'_{j'} > -k\}.$$
Observe
$$0\le \bold{1}_{B_{k}} R_{JJ'}= \bold{1}_{B_{k}} R_{J J' k}$$
for every $J,J',k$.  Put $\tilde{R}_{JJ'}:= \bold{1}_{X \backslash A} R_{JJ'}$. The last current is closed  positive.  Using  the fact that $\{\psi_{k} \not =0\} \subset B_{k} \backslash A$, we get 
\begin{align}\label{eq-bieudienRjkquaRjnew}
\psi_{k} \tilde{R}_{JJ'}= \psi_{k} R_{JJ'}= \psi_{k} R_{JJ'k}.
\end{align}
Put $p':=n- |J|-|J'|-p-1$.  By Claim in the proof of \cite[Proposition 4.2]{Viet-generalized-nonpluri}, for every $j'' \in \{1, \ldots, m\} \backslash (J \cup J')$ and every closed smooth form $\Phi$ of  bi-degree $(p',p')$ on $X$,  we have 
\begin{align}\label{limit-wjkbichanbansua}
\lim_{k \to \infty}\int_X \psi_{k} \ddc w_{j'' k} \wedge R_{JJ'k} \wedge \Phi =0.
\end{align}
%Similar equalities also hold if we replace $T$ by $T'$ in the defining formulae of $R_{JJ'}$ and $R_{JJ'k}$.   

Let 
$$S_0:= \langle T_{1} \wedge \cdots  \wedge T_{n}\dot{\wedge} T \rangle - \langle T'_{1} \wedge \cdots  \wedge T'_{n}\dot{\wedge} T \rangle$$
and 
$$S_1:= \langle T_{1} \wedge \cdots  \wedge T_{n}\dot{\wedge} T \rangle - \langle T'_{1} \wedge \cdots  \wedge T'_{n}\dot{\wedge} T \rangle, \quad S_2:= \langle T'_{1} \wedge \cdots  \wedge T'_{n}\dot{\wedge} (T- T') \rangle.$$
We have $S_0= S_1+ S_2$.  Using $T_{jk}= T'_{jk}+ \ddc w_{jk}$, one can check that
\begin{align*}
\int_X \psi_{k} S_1\wedge \Phi  =\sum_{s=1}^m \int_X \psi_{k} \wedge_{j=1}^{s-1} T'_{jk}\wedge \ddc w_{s k} \wedge \wedge_{j=s+1}^m T_{jk} \wedge T \wedge \Phi
\end{align*}
 for every closed smooth  $\Phi$. 
This together with (\ref{limit-wjkbichanbansua}) yields  
\begin{align}\label{limit-wjkbichanbansua2}
\langle S_1, \Phi \rangle = \lim_{k\to \infty} \langle \psi_{k}S_1, \Phi\rangle =0.
\end{align}
Let $\varphi_l:= \max\{\varphi, -l\}$ and $\varphi'_l:= \max\{\varphi', -l\}$ for $l \in \N$. By \cite[Theorem 2.2]{Viet-generalized-nonpluri},  observe
\begin{align} \label{eq-tinhS2phailayvarpgil}
 \int_X \psi_k S_2 \wedge \Phi=  \lim_{l \to \infty} \int_X \psi_k \ddc(\varphi_l - \varphi'_l) \wedge   T'_{1k} \wedge \cdots T'_{mk} \wedge \Phi.
\end{align}
Since $\varphi_l- \varphi'_l$ is bounded uniformly in $l \in \N$,  reasoning as in the proof of (\ref{limit-wjkbichanbansua}), we see that the term under limit in the right-hand side of (\ref{eq-tinhS2phailayvarpgil})   converges to $0$ as $k \to \infty$ uniformly in $l$.  Hence  
$$ \int_X \psi_k S_2 \wedge \Phi \to 0$$
as $k \to \infty$. Consequently,  we get $ \int_X \psi_k S \wedge \Phi \to 0$ as $k \to \infty$. In other words, (\ref{eq-same-singularity}) follows. This finishes Step 1. \\

\noindent
\textbf{Step 2.}  Consider now the general case, \emph{i.e,} $T'_j$ and $T'$ are less singular than $T_j$ and $T$ respectively. Without loss of generality, we can assume that $u'_j \ge u_j$ and $\varphi' \ge \varphi$.  For $l \in \N$, put $u_{j}^l:= \max\{ u_j, u'_j - l \}$ which is of the same singularity type as $u'_j$. Notice that $\ddc u_{j}^l+ \theta_j \ge 0$. Similarly, put $\varphi^l:= \max\{\varphi, \varphi'-l\}$ and  $T^l:= \ddc \varphi^l + \eta \ge 0$.  

Since $X$ is K\"ahler, the family of currents  $\langle \wedge_{j=1}^m (\ddc u_{j}^l+ \theta_j)  \dot{\wedge} T^l \rangle$ parameterized by $l$  is of uniformly bounded mass.  Let $S$ be a limit current of  the last family  as $l \to \infty$. Since $u^l_j, u'_j$ are of the same singularity type for every $j$ and $\varphi^l, \varphi'$ are so,  using  Step 1, we see that 
\begin{align}\label{eq-SbangT'phjay}
\{S\}= \{\langle \wedge_{j=1}^m T'_j  \dot{\wedge} T' \rangle\}.
\end{align}  %Since $\varphi^l$ decreases to $\varphi$ as $l\to \infty$, by \cite[Theorem 2.2]{Viet-generalized-nonpluri}, $(T_l)_l$ is a sequence of currents satisfying the condition $(*)$ introduced in \cite{Viet-generalized-nonpluri}. 
On the other hand, since $u_{j}^l, \varphi^l $ decrease to $u_j, \varphi$ as $l \to \infty$ respectively, we can apply \cite[Lemma 4.1]{Viet-generalized-nonpluri} (and  \cite[Theorem 2.2]{Viet-generalized-nonpluri}) to get
$$S \ge \langle \wedge_{j=1}^m T_j \dot{\wedge} T \rangle.$$
This combined with (\ref{eq-SbangT'phjay}) gives  the desired assertion. The proof is finished. 
\endproof

We note here the following remark which could be useful for other works.

\begin{remark} Let $P$ and $P'$ be closed positive $(1,1)$-currents and $Q$ a closed positive currents such that $P'$ is less singular than $P$ and potentials of $P$ are integrable with respect to the trace measure of $Q$. Put $T:= P \wedge Q$ and $T':= P' \wedge Q$.  Then Theorem \ref{th-mono-current11} still holds for these $T',T$ with the same proof. The only minor modification is that  the potentials $\varphi, \varphi'$ of $T,T'$  in the last proof  are replaced by those  of $P,P'$.
\end{remark}

For a $(1,1)$-current $P$, recall that \emph{the polar locus} $I_P$ of $P$ is the set of $x\in X$ so that the potentials of $P$ are equal to $-\infty$ at $x$. By abuse of language, we say that a closed positive current $T$ has no mass on a Borel set $A \subset X$, if the trace measure of $T$ has no mass on $A$. 

For every pseudoeffective $(1,1)$-class $\beta$ in $X$, we define its \emph{polar locus $I_\beta$} to be that of a current with minimal singularities in $\beta$. This is independent of the choice of a current with minimal singularities. We have the following.

\begin{lemma}  \label{le-relativeandnonrelative} Assume that  $T$ is of bi-degree $(1,1)$.  Then we have 
\begin{align}\label{eq-gennonpluriclassi3}
\langle T_1 \wedge \cdots \wedge T_m  \wedge T \rangle=\langle T_1 \wedge \cdots \wedge T_m \dot{\wedge} ( \bold{1}_{X \backslash I_{T}}T) \rangle,
\end{align}
In particular, $T$ has no mass on $I_T$, then 
$$\langle T_1 \wedge \cdots \wedge T_m  \wedge T \rangle=\langle T_1 \wedge \cdots \wedge T_m \dot{\wedge} T \rangle.$$
\end{lemma}

\proof By \cite[Proposition 3.6]{Viet-generalized-nonpluri}, we get 
\begin{align}\label{eq-gennonpluriclassi3proof}
\langle T_1 \wedge \cdots \wedge T_m  \wedge T \rangle= \bold{1}_{X \backslash I_{T}}\langle T_1 \wedge \cdots \wedge T_m \dot{\wedge} T \rangle.
\end{align}
Now using (\ref{eq-gennonpluriclassi3proof}) and  \cite[Proposition 3.5]{Viet-generalized-nonpluri} (vii) gives (\ref{eq-gennonpluriclassi3}).
%$$\langle T_1 \wedge \cdots \wedge T_m  \wedge T \rangle= \langle T_1 \wedge \cdots \wedge T_m \dot{\wedge} (\bold{1}_{X \backslash I_{T}} T) \rangle=\langle T_1 \wedge \cdots \wedge T_m \dot{\wedge}T \rangle $$because $T$ has no mass on $I_T$. 
This finishes the proof.
\endproof

Let $1 \le  l \le m$. Let $\alpha_l, \ldots,\alpha_m,\beta$ be pseudoeffective $(1,1)$-classes of $X$. Let $T_{j, \min}, T_{\min}$ be currents with minimal singularities in  the classes $\alpha_j, \beta$ respectively, where $l \le j \le m$.  By Theorem \ref{th-mono-current11}, the class 
$$\big\{\langle T_1 \wedge \cdots T_{l-1} \wedge  T_{l,\min} \wedge \cdots \wedge T_{m, \min} \dot{\wedge} T_{\min} \rangle \big \}$$   is  a well-defined pseudoeffective class which is independent of the choice of  $T_{\min}$ and $T_{j, \min}$ for $l \le j \le m$. We denote the last class by 
$$\big \{\langle T_1 \wedge \cdots \wedge T_{l-1} \wedge  \alpha_l\wedge  \cdots \wedge \alpha_m \dot{\wedge} \beta \rangle \big\}.$$
For simplicity, when $l=1$, we remove the bracket $\{ \quad  \}$ from the last notation. 

The following result holds for the class $\big \{\langle T_1 \wedge \cdots \wedge T_{l-1} \wedge  \alpha_l\wedge  \cdots \wedge \alpha_m \dot{\wedge} \beta \rangle \big\}$  but to avoid cumbersome notations (while keeping the essence of the statements), we only write it for $l=1$. 
  
\begin{proposition} \label{pro-productclasss}

$(i)$ The product  $\langle  \wedge_{j=1}^m \alpha_j \dot{\wedge} \beta \rangle$ is symmetric and   homogeneous  in $\alpha_1,$ $\ldots, \alpha_m$.  

$(ii)$ If $\beta'$ is a pseudo-effective $(1,1)$-class, then 
$$\langle \wedge_{j=1}^m \alpha_j \dot{\wedge} \beta \rangle + \langle \wedge_{j=1}^m \alpha_j \dot{\wedge} \beta' \rangle \le \langle \wedge_{j=1}^m \alpha_j \dot{\wedge} (\beta+\beta') \rangle.$$

$(iii)$ Let $1 \le l \le m$ be an integer. Let  $\alpha''_1, \ldots, \alpha''_l$ be a pseudoeffective $(1,1)$-class such that $\alpha''_j \ge \alpha_j$ for $1 \le j \le l$.  Assume that there is a current with minimal singularities in $\beta$ having no mass on $I_{\alpha''_j- \alpha_j}$ for every $1 \le j \le l$. Then, we have 
$$\langle \wedge_{j=1}^l \alpha''_j \wedge  \wedge_{j=l+1}^m \alpha_j \dot{\wedge} \beta \rangle \ge  \langle \wedge_{j=1}^m \alpha_j \dot{\wedge} \beta \rangle.$$

$(iv)$ If there is a current with minimal singularities in $\beta$ having no mass on proper analytic subsets on $X$, then the product   $\{\langle  \wedge_{j=1}^m \alpha_j \dot{\wedge} \beta \rangle\}$ is continuous on the set of $(\alpha_1,\ldots, \alpha_m)$ such that $\alpha_1, \ldots, \alpha_m$ are big.

$(v)$ We have 
$$\langle  \wedge_{j=1}^m \alpha_j  \wedge \beta \rangle  \le \langle  \wedge_{j=1}^m \alpha_j \dot{\wedge} \beta \rangle$$
and the equality occurs if there is a current with minimal singularities $P$ in $\beta$ such that $P=0$ on $I_P$. 
\end{proposition}

\proof We see that $(v)$ is a direct consequence of Lemma \ref{le-relativeandnonrelative} and the definition of the product $\langle  \wedge_{j=1}^m \alpha_j \dot{\wedge} \beta \rangle$. The other desired statements can be proved by using arguments similar to those in the proof of \cite[Proposition 4.6]{Viet-generalized-nonpluri}; see also \cite{Boucksom-derivative-volume} for related materials. This finishes the proof.  
\endproof

The following result will be useful later. 

\begin{lemma}\label{le-sosanhpolarlocus} Let $\alpha$ be a big class and let $T_{\alpha,\min}$ be a current with minimal singularities in $\alpha$. Let $T$ be a current in $\alpha$. Then, the current $T_\alpha:= \bold{1}_{I_{T_{\alpha, \min}}}T_{\alpha,\min}$ is a linear combination of currents of integration along irreducible hypersurfaces of $X$, and we have 
\begin{align}\label{ine-soanshpolarlocusalphamin}
T_{\alpha} \le \bold{1}_{I_{T}}T. 
\end{align}
In particular, for every pluripolar set $A$, if  $T$ has no mass on $A$, then  neither does $T_{\alpha,\min}$. 
\end{lemma}

\proof Recall that $I_\alpha= I_{T_{\alpha,\min}}$.  By Demailly's analytic approximation of $(1,1)$-currents (\cite{Demailly_analyticmethod}), there exists a K\"ahler current with analytic singularities $P$ in $\alpha$. It follows that $I_{\alpha}$ is contained in a proper analytic subset $V$ of $X$. This together the fact that $\supp T_{\alpha}$ is contained in the closure of $I_{\alpha}$ implies that $T_{\alpha}$ is supported on $V$. 

Since $T_{\alpha}$ is of bi-dimension $(n-1,n-1)$, using the first support theorem \cite[Page 141]{Demailly_ag}, we see that $T_{\alpha}$ is supported on the union of hypersurfaces of $X$ contained in $V$. Now the second support theorem \cite[Page 142-143]{Demailly_ag} impplies that $T_{\alpha}$ must be a linear combination of currents of integration along hypersurfaces. Hence    the first desired assertion follows.

We prove (\ref{ine-soanshpolarlocusalphamin}). It is enough to consider the case where $\bold{1}_{I_{T_{\alpha, \min}}}T_{\alpha,\min}$ is nonzero. Let $W$ be the support of the last current. By the above observation, $W$ is a hypersurface. Since $T$ is less singular than $T_{\alpha,\min}$, we get 
$$\nu(T, x) \ge \nu(T_{\alpha,\min},x)$$
 for every $x$. In particular, the generic Lelong number of $T$ along every irreducible component  $W'$  of $W$ is greater than or equal to that of $T_{\alpha, \min}$ along $W'$.   We deduce that  $T \ge \bold{1}_{I_{T_{\alpha, \min}}}T_{\alpha,\min}$. Hence,  (\ref{ine-soanshpolarlocusalphamin}) follows.

Let $A$ be a pluripolar set in $X$. Let $\varphi_{\min}$ be a potential of $T_{\alpha,\min}$. 
We have 
$$T_{\alpha,\min}= \bold{1}_{\{\varphi_{\min}> -\infty\}} T_{\alpha,\min}+ \bold{1}_{\{\varphi_{\min}= -\infty\}} T_{\alpha,\min}.$$
Denote by $I_1,I_2$ the first and second term in the right-hand side of the last equality respectively. By (\ref{ine-soanshpolarlocusalphamin}) and the hypothesis, we see that $I_2$ has no mass on $A$. We now show that $I_1$ satisfies the same property. 

If $\{\varphi_{\min}> -\infty\}$ is open, then it is clear that $I_1$ has no mass on $A$ because $\varphi_{\min}$ is locally bounded on the open set $\{\varphi_{\min}> -\infty\}$. However in general, when $\{\varphi_{\min}> -\infty\}$ is not necessarily open, some more arguments are needed. Recall that $I_1$ is actually equal to the non-pluripolar product $\langle T_{\alpha,\min} \rangle$ of $T_{\alpha,\min}$ itself (e.g, by applying \cite[Proposition 3.6 (i)]{Viet-generalized-nonpluri} to $T \equiv 1$ and $m=1$). Since the current $\langle T_{\alpha,\min} \rangle$ has no mass on pluripolar sets, we see that $I_1$ has no mass on $A$. Hence, $T_{\alpha, \min}$ has no mass on $A$. 
%Observe that the first term in the right-hand side of the last equality has no mass on $A$ (it is equal to the non-pluripolar product of $T_{\alpha,\min}$ itself), whereas the second term satisfies the same property by (\ref{ine-soanshpolarlocusalphamin}) and the hypothesis.  
 This finishes the proof. 
\endproof

We note that (\ref{ine-soanshpolarlocusalphamin}) actually holds in a much more general setting; see \cite[Lemma 4.1]{AhagCegrellHiepMA}.

\section{Proof of Theorems \ref{the-main} and \ref{the-main3}} \label{sec-themain12}

We will sometimes use the notations $\gtrsim, \lesssim$ to denote the inequalities $\ge, \le$ modulo some strictly positive multiplicative constant independent of parameters in consideration. 
For every analytic set $W$ in a complex manifold $Y$, we denote by $[W]$ the current of integration along $W$.  

Let $X$ be a compact K\"ahler manifold. Let $\alpha_1,\ldots,\alpha_m$ be big classes in $X$.    Let $T_{j,\min}$ be a current with minimal singularities in $\alpha_j$ and  
$$T_{\alpha_j}:= \bold{1}_{I_{\alpha_j}} T_{j,\min}$$
 (recall here that $I_{\alpha_j}$ is the set of $x \in X$ such that potentials of $T_{j,\min}$ are equal to $-\infty$ at $x$). By Lemma \ref{le-sosanhpolarlocus}, the current $T_{\alpha_j}$ is a linear combination of currents of integration along irreducible hypersurfaces of $X$.   In view of proving Theorem \ref{the-main},  we first explain how to reduce the problem to the case where $T_{\alpha_j}$'s are zero.

\begin{lemma}\label{le-trudiphananalytic}For every $j$, the class $\alpha_j - \{T_{\alpha_j}\}$ is big and there holds
\begin{align}\label{eq-wedgealphajtrudianalytic}
\langle \wedge_{j=1}^m \alpha_j \rangle = \langle \wedge_{j=1}^m (\alpha_j - \{T_{\alpha_j}\}) \rangle.
\end{align}
\end{lemma}

\proof  Let $\omega$ be a K\"ahler form on $X$. Fix an index $1 \le j \le m$.  Let $W_j$ be the support of $T_{\alpha_j}$. Consider a K\"ahler current $P_j \in \alpha_j$.  By Lemma \ref{le-sosanhpolarlocus}, the set $W_j$ is a hypersurface (or empty), and $P_j- T_{\alpha_j}$ is a closed positive current. Note that 
$$P_j- T_{\alpha_j} = P_j \gtrsim \omega$$
on $X \backslash W_j$. Since $\omega$ is smooth, we get   $P_j- T_{\alpha_j} \gtrsim \omega$ on $X$. In other words, $P_j- T_{\alpha_j}$ is a K\"ahler current. Hence, $\alpha_j - \{T_{\alpha_j}\}$ is big.

It remains to prove (\ref{eq-wedgealphajtrudianalytic}). The inequality direction $``\ge "$ is clear because $\alpha_j \ge \alpha_j - \{T_{\alpha_j}\}$. To get the converse inequality, one only needs to notice that 
$$\langle \wedge_{j=1}^m T_{j,\min} \rangle = \langle \wedge_{j=1}^m (T_{j,\min}- T_{\alpha_j}) \rangle$$
which is true because both sides are currents which have no mass on 
$$W:=\cup_{j=1}^m W_j$$
 (which is a closed pluripolar set) and are equal on $X \backslash W$ (which is an open subset of $X$).  The proof is finished. 
\endproof

Let $T_j \in \alpha_j$ be a closed positive current as in Theorem \ref{the-main}. By Lemma \ref{le-sosanhpolarlocus},  we have  $\bold{1}_{I_{T_j}}T_j \ge  T_{\alpha_j}$. It follows that  $T_j- T_{\alpha_j}$ is positive. Using the fact that $T_{\alpha_j}$ is supported on proper analytic subsets on $X$ gives  
$$\langle \wedge_{j=1}^m  T_j \rangle = \langle \wedge_{j=1}^m (T_j- T_{\alpha_j}) \rangle.$$
This combined with Lemma \ref{le-trudiphananalytic} yields that  $(T_1- T_{\alpha_1}), \ldots, (T_m-  T_{\alpha_m})$ are of full mass intersection. Hence, by considering $T_j- T_{\alpha_j}$, $\alpha_j - \{T_{\alpha_j}\}$ instead of $T_j, \alpha_j$, \emph{we can assume, from now on, that $T_{\alpha_j}$ is zero as desired}.

Assume for the moment that  $V$ is a  smooth submanifold of $X$ of dimension $\le n-1$. Let $\sigma: \widehat X \to X$ be  the blowup of $X$ along $V$. Denote by $\widehat V$ the exceptional hypersurface. Let $\omega$ be a K\"ahler form on $X$.  Let $\omega_h$ be a closed smooth form cohomologous to $-[\widehat V]$ so that the restriction of $\omega_h$ to each fiber of the natural projection from $\widehat V$ to $V$ is strictly positive (the existence of such a form is classical, see \cite[Lemma 3.25]{Voisin1}).  Thus, there exists a strictly positive constants $c_V$ satisfying that 
\begin{align}\label{ine-wideharomegaeta}
\widehat  \omega:= c_V \sigma^* \omega+ \omega_h>0
\end{align}
We note that when $\dim V = n-1$, by convention, we put $\widehat X:=X$, $\sigma:=\id$, $\widehat V:=V$, $c_V:=1$ and $\omega_h:= 0$. 
 
For every closed positive current $S$ on $X$, let $\lambda_{S}$ be the generic Lelong number of $S$ along $V$. By a well-known result on Lelong numbers under blowups (see \cite[Corollary 1.1.8]{Boucksom-these}), the generic Lelong number of $\sigma^* S$ along $\widehat V$ is equal to $\lambda_S$.  Hence, we can decompose 
$$\sigma^* T_j= \lambda_{T_j} [\widehat V]+ \eta_j,\quad \sigma^* T_{j,\min}= \lambda_{T_{j,\min}} [\widehat V]+ \eta_{j,\min},$$
where $\eta_j$ and  $\eta_{j,\min}$ are currents whose generic Lelong numbers along $\widehat V$ are zero. Since $T_{j,\min}$ is less singular than $T_j$, we have $\lambda_{T_j}\ge \lambda_{T_{j,\min}}$. 

Let 
$$\gamma_j:= \{\eta_j\}, \quad \gamma_{j,\min}:= \{\eta_{j,\min}\}, \quad \beta:= \{ [\widehat V]\}.$$
These classes are important in the sequel.  By \cite{Boucksom-volume,Demailly-Paun}, the class $\gamma_{j,\min}$ is big.  
For every closed smooth $(n-m,n-m)$-form  $\Phi$, using the fact that  $T_{j,\min}$ has minimal singularities and the monotonicity of non-pluripolar products gives 
\begin{align} \label{eq-Tminclassgamma}
\int_X \langle \wedge_{j=1}^m T_{j,\min} \rangle \wedge \Phi=\int_{\widehat X} \langle \wedge_{j=1}^m \eta_{j,\min} \rangle \wedge \sigma^*\Phi= \int_{\widehat X}\langle \wedge_{j=1}^m \gamma_{j,\min} \rangle  \wedge \sigma^*\Phi.
\end{align}

\begin{lemma} \label{le-sosanhetaSsphay} We have 
\begin{align} \label{ine-etaSS}
\langle \wedge_{j=1}^m \eta_{j} \rangle \le \langle \wedge_{j=1}^{m-1} \eta_j \dot{\wedge} \eta_{m} \rangle, \quad \langle \wedge_{j=1}^m \eta_{j,\min} \rangle = \langle \wedge_{j=1}^{m-1}\eta_{j,\min} \dot{\wedge} \eta_{m,\min} \rangle,
\end{align}
and 
\begin{align} \label{ine-etaSS22}
\langle \wedge_{j=1}^m \gamma_{j,\min} \rangle = \langle \wedge_{j=1}^{m-1} \gamma_{j,\min} \dot{\wedge}  \gamma_{m,\min} \rangle
\end{align}
\end{lemma}

\proof The first desired inequality in (\ref{ine-etaSS}) is clear by  Proposition \ref{pro-productclasss}.  Observe that   $\bold{1}_{I_{\eta_{m,\min}}} \eta_{m,\min}$ has no mass on $\widehat V$ because the generic Lelong number of $\eta_{m,\min}$ along $\widehat V$ is equal to zero. We deduce that
$$\bold{1}_{I_{\eta_{m,\min}}}\, \eta_{m,\min}= \bold{1}_{I_{\eta_{m,\min}} \backslash \widehat V}\, \eta_{m,\min} \le \sigma^* \big(\bold{1}_{\sigma(I_{\eta_{m,\min}})} T_{m,\min}\big) \le \sigma^*(\bold{1}_{I_{T_{m,\min}}} T_{m,\min}) =0.$$
Hence, $\eta_{m,\min}$ has no mass on $I_{\eta_{m,\min}}$. Combining this with  Lemma \ref{le-relativeandnonrelative} yields (\ref{ine-etaSS}).

We now prove (\ref{ine-etaSS22}). Let $Q_m$ be a current with minimal singularities in $\gamma_{m,\min}$. By Lemma \ref{le-sosanhpolarlocus} and the fact that $\gamma_{m,\min}$ is big, we see that 
$$\bold{1}_{I_{Q_m}} Q_m \le  \bold{1}_{\eta_{m,\min}} \eta_{m,\min}=0.$$
Hence,  $Q_m$ has no mass on $I_{Q_m}$. Using this and Lemma \ref{le-relativeandnonrelative} gives the desired equality and finishes the proof. 
\endproof

Fix a norm $\| \cdot \|$ on $H^{1,1}(X, \R)$. For $1 \le j  \le m$, let $P_j$ be a K\"ahler current with analytic singularities in $\alpha_j$.   Let $\epsilon>0$ be a constant small enough so that  $P_j \ge \epsilon \omega$ for every $1 \le j\le m$.  

\begin{lemma} \label{le-dongkahlertronfeta} For every constant $\delta \in (0,1)$, there exist a  constant $c_\delta>0$ and a K\"ahler current with analytic singularities $Q_j \in \gamma_{j,\min}- c_\delta \beta$ for $1 \le j \le m$ such that $I_{Q_j}$ does not contain $\widehat V$, and $Q_j \ge \frac{\delta \epsilon}{2 c_V} \widehat \omega$, and 
\begin{align}\label{ine-uocluongcdelta}
\frac{\delta \epsilon}{2 c_V}\le c_\delta \le \big(c\|\alpha_j\|+ \frac{\epsilon}{2 c_V}\big) \delta,
\end{align}
for some constant $c>0$ independent of $\delta, \beta$ and $\alpha_j$.  In particular, the currents with minimal singularities in $\gamma_{j,\min}- c_\delta \beta$ has no mass on $\widehat V$, and the current $[\widehat V]$ has no  mass on the polar locus of the class $\gamma_{j,\min}- c_\delta \beta- \frac{\delta \epsilon}{2 c_V} \{\widehat \omega\}$. 
\end{lemma}

\proof  
Using Demailly's analytic approximation of currents (\cite{Demailly_analyticmethod}) applied to the K\"ahler current $(1-\delta)T_{j,\min}+ \delta P_j$ for $\delta \in (0,1)$, we obtain that for every $\delta \in (0,1)$, there exits a K\"ahler current $P_{j,\delta}$ with analytic singularities in the class $\alpha_j$ such that $P_{j,\delta}$ is less singular than $(1-\delta)T_{j,\min}+ \delta P_j$ and 
\begin{align}\label{ine-PJdelta}
P_{j,\delta} \ge \delta \epsilon \omega/2.
\end{align}
We deduce that 
\begin{align}\label{ine-sosanhPjdelta}
\lambda_{T_{j,\min}} \le \lambda_{P_{j,\delta}} \le \lambda_{T_{j,\min}}+ a_j \delta,
\end{align}
where $a_j:= \lambda_{P_j}- \lambda_{T_{j,\min}} \ge 0$.  Write 
$$\sigma^* P_{j,\delta}=  \lambda_{P_{j,\delta}} [\widehat V]+ \eta_{j,\delta}.$$
 Since $P_{j,\delta}$ has analytic singularities, so does $\eta_{j,\delta}$ and the polar locus of $\eta_{j,\delta}$ is an analytic subset of $X$ which doesn't contain $\widehat V$.  Hence, $[\widehat V]$ has no mass on the polar locus of $\eta_{j,\delta}$. 

Recall that by the choice of $\omega_h$, we have $\omega_h \in -\beta$. By (\ref{ine-PJdelta}) and (\ref{ine-wideharomegaeta}), we also get 
$$Q_j:=\eta_{j,\delta}+ \frac{\delta\epsilon}{2 c_V} \omega_h \ge \frac{\delta \epsilon}{2 c_V} \widehat \omega.$$
 The last current is in the class 
$$\gamma_{j,\min} - c_\delta \beta,$$
where 
$$c_\delta:= \lambda_{P_{j,\delta}} -\lambda_{T_{j,\min}}+ (\delta \epsilon)/(2c_V) \le \big(\lambda_{P_{j}} -\lambda_{T_{j,\min}}+ \epsilon/(2c_V)\big) \delta$$
by (\ref{ine-sosanhPjdelta}).  Since $P_{j}$ is a current in $\alpha_j$, we get  $\lambda_{P_{j}} \le c\|\alpha_j\|$ for some positive constant $c$ independent of $\alpha_j$. Hence, (\ref{ine-uocluongcdelta}) follows.

We have proved that there is a K\"ahler current with analytic singularities $Q_j$ in $\gamma_{j,\min}- c_\delta \beta$ such that $\widehat V \not \subset I_{Q_j}$. It follows that $Q_j$ has no mass on $\widehat V$. Using this and Lemma \ref{le-sosanhpolarlocus} yields that the currents with minimal singularities in $\gamma_{j,\min}- c_\delta \beta$ has no mass on $\widehat V$.
The last desired assertion is also immediate because the polar locus of $Q_j- \frac{\delta \epsilon}{2 c_V} \widehat \omega$ does not contain $\widehat V$. This finishes the proof.
\endproof

\begin{proof}[End of the proof of Theorem \ref{the-main}]   Let 
$$b_j:= \lambda_{T_j}- \lambda_{T_{j,\min}}\ge 0.$$
Note that $\gamma_j= \gamma_{j,\min}  - b_j \beta$. Suppose on contrary that  $b_j >0$ for every $j$.   Recall that we are assuming that $V$ is smooth. The case where $V$ is singular is dealt with later. 

Let $c_\delta$ be the constant associated to a number $\delta \in (0,1)$ as in Lemma \ref{le-dongkahlertronfeta}. Let $c$ be the constant appearing in  (\ref{ine-uocluongcdelta}).  Put 
$$\delta_j:= \big(c\|\alpha_j\|+ \frac{\epsilon}{2 c_V}\big)^{-1}b_j $$
for $1 \le j \le m$. Note that since $b_j \lesssim \|\alpha_j\|$, we can increase $c$ in order to have  $\delta_j \in (0,1)$.  By (\ref{ine-uocluongcdelta}), we get $c_{\delta_j} \le b_j$ for every $j$. Let $\gamma'_{j,\min}:= \gamma_{j,\min}  - c_{\delta_j} \beta$. By Lemma \ref{le-dongkahlertronfeta} and the fact that 
$$I_{\gamma'_{j,\min}- \gamma_{j}}= I_{(b_j-c_{\delta_j})\beta} \subset \widehat V,$$
we obtain that  the currents with minimal singularities in  $\gamma'_{m,\min}$ has no mass on $I_{\gamma'_{j,\min}- \gamma_{j}}$. This combined with Proposition \ref{pro-productclasss} $(iii)$ gives
\begin{align*}
\{ \langle \wedge_{j=1}^m \eta_j \rangle\}  \le \langle \wedge_{j=1}^{m-1}\gamma_j  \dot{\wedge} \gamma_m \rangle  \le \langle  \wedge_{j=1}^{m-1}\gamma_{j} \dot{\wedge}\gamma'_{m,\min} \rangle \le \langle  \wedge_{j=1}^{m-1}\gamma'_{j,\min} \dot{\wedge}\gamma'_{m,\min} \rangle.
\end{align*}
Using  the supper-additivity of products of classes (Proposition \ref{pro-productclasss} $(ii)$), we get
\begin{align*}
\langle  \wedge_{j=1}^{m-1}\gamma'_{j,\min} \dot{\wedge}\gamma'_{m,\min}  \rangle \le\langle  \wedge_{j=1}^{m-1}\gamma'_{j,\min} \dot{\wedge}\gamma_{m,\min} \rangle- c_{\delta_m} \langle  \wedge_{j=1}^{m-1}\gamma'_{j,\min} \dot{\wedge} \beta \rangle
\end{align*}
Let $I$ be the first term in the right-hand side in the last inequality. Recall that the currents with minimal singularities  in $\gamma_{m,\min}$ has no mass on $\widehat V$. The last set contains $I_{\beta}$. Hence, using Proposition \ref{pro-productclasss} $(iii)$ implies 
$$I \le \langle  \wedge_{j=1}^{m-1}\gamma_{j,\min} \dot{\wedge}\gamma_{m,\min} \rangle.$$
 Consequently, 
\begin{align*}
\langle  \wedge_{j=1}^{m-1}\gamma'_{j,\min} \dot{\wedge}\gamma'_{m,\min}  \rangle &\le  \langle  \wedge_{j=1}^{m-1}\gamma_{j,\min} \dot{\wedge}\gamma_{m,\min} \rangle- c_{\delta_m} \langle  \wedge_{j=1}^{m-1}\gamma'_{j,\min} \dot{\wedge} \beta \rangle\\
& \le\langle  \wedge_{j=1}^{m}\gamma_{j,\min} \rangle - c_{\delta_m} \langle  \wedge_{j=1}^{m-1}\gamma'_{j,\min} \dot{\wedge} [\widehat V] \rangle
\end{align*}
by Lemma \ref{le-sosanhetaSsphay}. Now let $\Phi$ be a closed smooth positive $(n-m,n-m)$-form on $X$. Put $M_j:= \frac{\delta_j \epsilon}{2 c_V}$. Note that by (\ref{ine-uocluongcdelta}), we get $M_j \le c_{\delta_j}$ for every $j$. Taking into account Lemma \ref{le-dongkahlertronfeta} and Proposition \ref{pro-productclasss} $(iii)$, we see that 
\begin{align*}
\int_{\widehat X}\langle  \wedge_{j=1}^{m-1}\gamma'_{j,\min} \dot{\wedge} [\widehat V] \rangle \wedge \sigma^* \Phi &\ge M_1 \cdots M_{m-1} \int_{\widehat V} \widehat \omega^{m-1} \wedge \sigma^* \Phi. %\\&=  M_1 \cdots M_{m-1} \int_{\widehat V} \omega_h^{m-1} \wedge \sigma^* \Phi= M_1 \cdots M_{m-1} \langle [V], \Phi \rangle.
\end{align*}
%by Fubini's theorem and the choice of $\omega_h$.   
Consequently, we obtain
\begin{align} \label{ine-danhgiagancuiocungTjMcdelta}
\int_X \langle \wedge_{j=1}^m T_j \rangle  \wedge \Phi &= \int_{\widehat X}  \langle \wedge_{j=1}^m \eta_j \rangle \wedge \sigma^*\Phi\\
\nonumber
& \le  \int_{\widehat X}\langle  \wedge_{j=1}^{m}\gamma_{j,\min} \rangle \wedge \sigma^* \Phi - M_1 \cdots M_{m} \langle [\widehat V] \wedge \sigma^*\Phi, \widehat \omega^{m-1} \rangle
\end{align}
which is, by (\ref{eq-Tminclassgamma}), equal to  
$$\int_X \langle \wedge_{j=1}^m T_{j,\min} \rangle \wedge \Phi - M_1 \cdots M_{m} \langle [\widehat V] \wedge \sigma^*\Phi, \widehat \omega^{m-1} \rangle.$$
Using this and the hypothesis that
\begin{align}\label{eq-PhiTjTjmin}
\int_X \langle \wedge_{j=1}^m T_j \rangle  \wedge \Phi= \int_X \langle \wedge_{j=1}^m T_{j,\min} \rangle  \wedge \Phi,
\end{align}
we infer that $[ \widehat V] \wedge  \sigma^*\Phi =0$ for every closed smooth $(n-m,n-m)$-form $\Phi$. The last property means that $[V] \wedge \Phi =0$ for every closed smooth $(n-m,n-m)$-form $\Phi$. By choosing $\Phi:= \omega^{n-m}$, we obtain  a contradiction because $\dim V \ge n-m$. This finishes Step 1 of the proof.  We observe that we didn't fully use the assumption that $\{ \langle \wedge_{j=1}^m T_j \rangle \}=\{ \langle \wedge_{j=1}^m T_{j,\min} \rangle \}$. We only needed that there is a closed positive smooth $(n-m,n-m)$-form $\Phi$ on $X$ such that (\ref{eq-PhiTjTjmin}) holds and $[V]\wedge \Phi \not = 0$. We will use this remark in the next paragraph.

We now explain how to treat the case where $V$ is not necessarily smooth. By Hironaka's desingularization, there is $\sigma': X' \to X$ which is a composition of  consecutive blowups along smooth centers starting from $X$ so that the centers don't intersect the regular part of $V$ and the strict transform $V'$ of $V$ by $\sigma'$ is smooth.  Note that $V'$ is of the same dimension as $V$.

Let $T'_j:= \sigma'^* T_j$ and $\alpha'_j:= \sigma'^* \alpha_j$. One should note that  $T'_1, \ldots, T'_m$ might not be  of full mass intersection, however, we still have 
\begin{align} \label{eq-alphaphayphayPhigiaikidi}
\int_{X} \langle \wedge_{j=1}^m T'_j \rangle \wedge \sigma'^* \Phi = \int_{X} \langle \wedge_{j=1}^m\alpha_j \rangle \wedge \Phi =\int_{X'} \langle \wedge_{j=1}^m\alpha'_j \rangle \wedge \sigma'^* \Phi,
\end{align}
for every closed smooth $(n-m,n-m)$-form $\Phi$ on $X$. We will use $\Phi:= \omega^{n-m}$. Observe that
\begin{align*} %\label{ine-VphayPhigiaikidi}
[V'] \wedge \sigma'^* \Phi \not = 0
\end{align*}
because $\sigma'$ is a biholomorphism on an open Zariski set containing the regular part of $V$ and $[V] \wedge \Phi \not =0$ (here we use $\dim V \ge n-m$).  This together with (\ref{eq-alphaphayphayPhigiaikidi}) and the observation at the end of Step 1 allows us to apply Step 1 to $X',\alpha'_j$ and $T'_j$ to obtain that there exist an index $j_0$ such that 
$$\nu(T'_{j_0},V')= \nu(\alpha'_{j_0},V').$$
 On the other hand, by construction of $\sigma'$, we get $\nu(T'_j, V')= \nu(T_j, V)$ for every $j$, a similar property also holds for $T_{j,\min}$.  It follows that 
$$\nu(T_{j_0},V)= \nu(\alpha'_{j_0},V') \le \nu(T'_{j_0,\min}, V')= \nu(T_{j_0,\min}, V) \le \nu(T_{j_0},V).$$ 
Hence,  we get $\nu(T_{j_0,\min}, V) = \nu(T_{j_0},V)$.
This finishes the proof.
\end{proof}

We now present the proof of Theorem \ref{the-main3}.

\begin{proof}[Proof of Theorem \ref{the-main3}] Let $\omega$ be a fixed K\"ahler form on $X$.  Observe that by homogeneity, in order to prove the desired inequality, it suffices to consider  $\alpha_j/\|\alpha_j\|$ in place of $\alpha_j$. Hence, from now on, without loss of generality, we can assume that $\alpha_j \in \cali{B}_0\cap\cali{S}$, where $\cali{S}$ is the unit sphere in $H^{1,1}(X,\R)$. Since $\cali{B}_0$ is closed and contained in the big cone, we deduce that $\cali{B}_0\cap \cali{S}$ is compact in the big cone. It follows that  there exist a constant $\epsilon>0$ such that for every $\alpha \in \cali{B}_0 \cap \cali{S}$, there exists a current with analytic singularities $P \in \alpha$ such that $P \ge \epsilon \omega$. In particular, we obtain currents with analytic singularities $P_j \in \alpha_j$ such that $P_j \ge \epsilon \omega$ for $1 \le j \le m$.  

Now, we follow the arguments in the proof of Theorem \ref{the-main}. One only needs to review carefully the constants involving in  estimates used there.     Our submanifold $V$ is now the point set $\{x_0\}$. Let the notations be as in the proof of Theorem \ref{the-main}.   %The blowup $\sigma: \widehat X \to X$ is obtained by blowing up $X$ along  $x_1,\ldots,x_s$ simultaneously. Let $\widehat V_l:= \sigma^{-1}(x_l)$ for $1 \le l \le s$.  Let $\omega_{h_l}$ be a closed smooth form cohomologous to $-[\widehat V_l]$ such that its restriction to each fibers of $\widehat V_l \to \{x_l\}$ is strictly positive, for $1 \le  l \le s$. Put $b_{jl}:= \nu(T_j, x_l)- \nu(\alpha_j, x_l)$ for $1 \le j \le n$ and $1 \le l \le s$. It suffices to consider the case where $b_{jl}>0$ for every $j$ and $l$.  Since $\alpha_j$ is in a compact $\cali{B}_0$ and $\nu(T_j, x_l) \lesssim \|\alpha_j\|$, we get that $b_{jl} \lesssim 1$.  
By the construction of $\widehat X$,  the constant $c_V>0$ in (\ref{ine-wideharomegaeta}) can be chosen to be independent of $x_0$. As in the proof of Theorem \ref{the-main}, put 
$$b_{j}:= \nu(T_j, x_0)- \nu(\alpha_j, x_0), \quad \delta_{j}:= \big(c\|\alpha_j\|+ \frac{\epsilon}{2 c_V}\big)^{-1}b_{j}, \quad M_{j}:= \frac{\delta_{j} \epsilon}{2 c_V}$$ 
for $1 \le j \le n$, where $c$ is a constant big enough depending only on $X$ (and a fixed K\"ahler form $\omega$ on $X$ and a fixed norm on $H^{1,1}(X,\R)$).   Since $\alpha_1,\ldots, \alpha_n \in \cali{B}_0 \cap\cali{S}$, we get  
$$\delta_{j} \gtrsim b_{j},$$
and the constant $\epsilon$ can be chosen independent of $\alpha_1,\ldots, \alpha_n$.  Using   (\ref{ine-danhgiagancuiocungTjMcdelta}) for $\Phi$ to be the constant function equal to $1$ gives
$$\int_X \big(\langle \wedge_{j=1}^n\alpha_j \rangle- \{\langle \wedge_{j=1}^n T_j \rangle\}\big) \ge   M_1 \cdots M_n= \frac{\delta_1 \epsilon}{2 c_V}\cdots \frac{\delta_n \epsilon}{2 c_V} \gtrsim b_1 \cdots b_n.$$
The proof is finished. 
\end{proof}

\begin{example}\label{ex-themain3} Let $Y$ be a compact K\"ahler manifold and $\theta$ be a semi-positive $(1,1)$-form in $Y$ such that there is a current $P$ in $\{\theta\}$ with $\nu(P,x_0)>0$ for some $x_0 \in Y$ (one can take, for example, $Y$ to be the complex projective space and $\theta$ to be its Fubini-Study form).  Let  $X:= Y^2$ and $\alpha:= \pi_1^*\{\theta\}$ which is a semi-positive class, where $\pi_1: Y^2 \to Y$ is the projection to the first component. We have $\int_X \alpha^{2 \dim Y}=0$. Hence, $\alpha$ is not big. Let $\omega$ be a K\"ahler form on $X$. Let $\alpha_\epsilon:= \alpha+ \epsilon \{\omega\}$. We have $$\int_X\alpha_\epsilon^{2 \dim Y} \to \int_X \alpha^{2 \dim Y} =0.$$
Hence, if the constant $C$ in Theorem \ref{the-main3} were independent of $\cali{B}_0$, then   (\ref{ine-truonghopncaliB0}) for $x_0$ would hold for $\alpha_j:= \alpha_\epsilon$ and $T_j= \pi_1^* P$  for every $j$ for some constant $C$ independent of $\epsilon$. Letting $\epsilon \to 0$ gives a contradiction because the left-hand side converges to $0$, whereas the right-hand side converges to  a positive constant.
\end{example}

\begin{proof}[Proof of Corollary \ref{cor-knownbignef}] 
We explain how to obtain Corollary \ref{cor-knownbignef} from Corollary \ref{cor-main2}.  Let $\rho: X' \to X$ be a smooth modification of $X$ and $E$ an irreducible hypersurface in $X'$. Let $\varphi':= \varphi \circ \rho$, $\varphi'_{\alpha,\min}:= \varphi_{\alpha,\min} \circ \rho$,   $\theta':= \rho^* \theta$ and $\alpha':= \rho^* \alpha$.  Since non-pluripolar products have no mass on pluripolar sets, we have
$$\langle (\ddc \varphi'+ \theta')^n \rangle = \langle \alpha'^n \rangle = \langle \alpha^n \rangle>0,$$ 
and a similar equality also holds if $\varphi'$ is replaced by $\varphi'_{\alpha,\min}$ (note that we don't know if  the latter is a quasi-psh function with minimal singularities in $\alpha'$; anyway we will only need that $\varphi'_{\alpha,\min}$ is of full Monge-Amp\`ere mass in $\alpha'$). By a well-known result in \cite{Boucksom-volume}, the class $\alpha'$ is big.  

Applying Corollary \ref{cor-main2} to $\ddc \varphi'+ \theta'$ and $V:= E$, we obtain that the generic Lelong number of $\varphi'$ along $E$ is equal to $\nu(\alpha',E)$. We also get an analogous property for $\varphi'_{\alpha,\min}$ by applying  Corollary \ref{cor-main2} to $\ddc \varphi'_{\alpha,\min}+ \theta'$. It follows that the generic Lelong numbers of $\varphi'$ and $ \varphi'_{\alpha,\min}$ along $E$ are equal. Now using this property and \cite[Corollary 10.18]{Boucksom_l2} (or \cite[Theorem A]{Boucksom-Favre-Jonsson}) gives the desired assertion. The proof is finished. 
\end{proof}

We end the paper with the example mentioned in Introduction. 

\begin{example} \label{ex-optimallelong} Let $X:= \P^n$ and $[x_0:x_1: \cdots:x_n]$ the homogeneous coordinates. Let $\omega$ be the Fubini-Study form on $\P^n$. Let $2 \le m \le n$ be an integer.   Consider $$V:= \big\{[x_0: \cdots:x_n] \in \P^n: x_j=0, \quad 0 \le j \le m-1\big\},$$ and 
$$T_j:= \ddc (|x_0|^2+ \cdots+ |x_{m-1}|^2)= \ddc \frac{|x_0|^2+ \cdots+ |x_{m-1}|^2}{|x_0|^2+ \cdots+ |x_{n}|^2}+ \omega$$
for $1 \le j \le m-1$. We have $\dim V = n-m$. Put $T_m:= \omega$. Observe that the currents $T:=T_1 \wedge \cdots \wedge T_m$ and $T':= T_1 \wedge \cdots \wedge T_{m-1}$ are well-defined (classically) by \cite[Corollary 4.11, Page 156]{Demailly_ag}. Moreover since $V$ is of dimension $n-m$, and $T'$ is of bi-dimension $(n-m+1,n-m+1)$, we see that the trace measure of $T'$ has no mass on $V$ by \cite[Page 141]{Demailly_ag}. This combined with the fact that $T_m$ is smooth yields that the trace measure of $T$ also has no mass on $V$. Using this and the fact that $T_j$ is smooth outside $V$, we obtain 
 $$T= \langle T_1 \wedge \cdots \wedge T_m\rangle$$
 (both sides have no mass on $V$). It follows that $T_1,\ldots, T_m$ are of full mass intersection, but  $\nu(T_j,V)>0$ for $1 \le j \le m-1$, and $\nu(T_m, V)= 0$.       
\end{example}

\bibliography{biblio_family_MA,biblio_Viet_papers}
\bibliographystyle{siam}

\bigskip

\noindent
\Addresses
\end{document}